\documentclass[a4paper, 11pt]{article}
\usepackage[utf8]{inputenc}
\usepackage{geometry}
\usepackage{graphicx}
\usepackage{authblk} %
\usepackage{amsmath}
\usepackage{amssymb}
\usepackage{amsthm}

\usepackage{mathrsfs} %
\usepackage{stmaryrd} %

\usepackage{xcolor}%
\usepackage{mathtools}

\usepackage{csquotes} %
\usepackage{hyperref}
\hypersetup{colorlinks,citecolor=blue,filecolor=blacklinkcolor=red, urlcolor=magenta,linkcolor=red}

 \usepackage[normalem]{ulem} %
\usepackage{tikz,tikz-cd} %

\newtheorem{lemm}{Lemma}

\newtheorem{theo}{Theorem}
\newtheorem{prop}{Proposition}

\theoremstyle{remark}

\newtheorem*{rema*}{Remark}

\usepackage[style=ieee,hyperref,backend=biber,sorting=nty]{biblatex}
\AtEveryBibitem{%
  \clearfield{issn} %
  \clearfield{month}
  \clearlist{language}
  \ifentrytype{online}{}{%
    \clearfield{url}
  }
}

\newcommand{\N}{\mathbb{N}}
\newcommand{\Z}{\mathbb{Z}}
\newcommand{\R}{\mathbb{R}}
\newcommand{\C}{\mathbb{C}}

\newcommand{\Mbar}{\overline{M}}

\title{Inverse problems for the Bakry-\'Emery Laplacian on manifolds with boundary - uniqueness and non-uniqueness} 
\author{Jack Borthwick\thanks{\href{mailto:jack.borthwick@mcgill.ca}{jack.borthwick@mcgill.ca}}, Niky Kamran\thanks{\href{mailto:niky.kamran@mcgill.ca}{niky.kamran@mcgill.ca}}}
\affil{McGill University\\}
\date{March 26, 2025}

\begin{document}
\maketitle
{\let\thefootnote\relax}
\begin{abstract}
We study the questions of uniqueness and non-uniqueness for a pair of closely related inverse problems for the Bakry-\'Emery Laplacian $-\Delta_{\mathcal E}$ on a smooth compact and oriented Riemannian manifold with boundary $(\overline{M},g)$, endowed with a volume form $\mathfrak{m}=e^{-V}\omega_g$. These consist in recovering the Taylor coefficients of metric $g$ and weight $V$ along the boundary of $\overline{M}$ from the knowledge of a pair of operators that can be viewed as geometrically natural Dirichlet-to-Neumann maps associated to $-\Delta_{\mathcal E}$. 
\end{abstract}

\section{Introduction}

We consider a manifold with boundary endowed with a smooth Bakry-\'Emery structure, that is a smooth compact and oriented Riemannian manifold with boundary $(\overline{M},g)$ of dimension $n\geq 3$, with interior $M$ and boundary $\partial M$, endowed with a volume form $\mathfrak{m}$ given by
 \[\mathfrak{m}=e^{-V}\omega_g\,,\]
where $\omega_g$ denotes the standard Riemannian volume form canonically induced on $\overline{M}$ by the metric $g$ and $V$ is a smooth function on $\overline{M}$. 

Following~\cite{Ambrosio_Gigli_Savare_2015}, we introduce the quadratic form $\mathcal{E}$ defined on $C_0^\infty(M)$ by 
\begin{equation}
    \label{eq:quadratic}
\mathcal{E}(u) = \int_M g(\textrm{d}u,\textrm{d}u) \mathfrak{m},\end{equation} and the associated\footnote{The assumption $u \in C_0^\infty(M)$ allows us to drop the boundary terms when performing integration by parts in the computation of the first variation of $\mathcal E$.} Bakry-Emery Laplacian $-\Delta_{\mathcal E}$\,, defined by
\begin{equation}\label{eq:operator_ambrosio}-\Delta_{\mathcal E}=-\Delta_g + g(\textrm{d}V,\cdot)\,. \end{equation}
Although the quadratic form $\mathcal E$ is defined on functions $u$ with compact support in the interior of $\Mbar$, the coefficients of $\Delta_{\mathcal E}$, depending only on the metric $g$, the weight $V$ and their derivatives, are well defined up to the boundary $\partial M$. 

The object of this article is the study of a pair of closely related inverse problems for the operator $-\Delta_{\mathcal E}$, both of which stem from geometrically natural definitions of operators that can be viewed as Dirichlet-to-Neumann (DN) maps associated to $-\Delta_{\mathcal E}$. 

Before defining these maps, we observe that from a geometric point of view, it is an intriguing choice to decouple the volume form of the manifold from the metric $g$ through the introduction of a weight $V$. It is therefore natural to ask if there is some geometric context in which such a decoupling may be canonical.

Our answer to this question is based on the observation~\cite{Gonzalez-Lopez_Kamran_Olver_1994} that any second-order differential operator can be rewritten as a Schrödinger operator:
\begin{equation} \mathcal{S}=-g^{ab}D^\omega_a\,D_b^\omega + U, \end{equation}
acting on sections of the rank $1$ vector bundle: \begin{equation}\label{def:G1} \mathscr{G}[1]=P\times_{\R} \R\,,\end{equation} associated to a principal bundle $\pi: P\simeq \Mbar\times \R \rightarrow \Mbar$ with group $(\R,+)$. The representation of $(\R,+)$ defining $\mathscr{G}[1]$ is simply: \[\mu \mapsto e^\mu\,,\]
and $P$ is equipped with a principal connection $\omega$ that induces a linear connection $D^\omega$ on $\mathscr{G}[1]$.
We can define, in a similar fashion, a line bundle $\mathscr{G}[w]$ using the representation $\mu \mapsto e^{w\mu}$ for any $w\in \R$. For our purposes, we observe that: \[ \mathscr{G}[-1]\simeq \mathscr{G}[1]^*, \quad \mathscr{G}[-2]=\mathscr{G}[1]^*\otimes\mathscr{G}[1]^*.\]  

It is easily seen that if $\mathcal{S}$, evaluated along some section $s: \Mbar \rightarrow P$, takes the form $-\Delta_\mathcal{E}$, then the principal connection must be flat, i.e. $d\omega =0$, and
\begin{equation}\label{eq:defU} U=-\frac{1}{2}\Delta_gV + \frac{1}{4}g(\textrm{d}V,\textrm{d}V)\,.\end{equation}
A consequence of this is that one can find a parallel metric $\zeta \in \Gamma \,\big(\mathscr{G}[-2]\big)$, determined up to a multiplicative constant, enabling us to introduce a quadratic form $\hat {\mathcal{E}}$ acting on sections $\phi\in\Gamma
\big(\mathscr{G}[1]\big)$ given by:
\[ \hat{\mathcal{E}}(\phi)=\int_M (g^{ab}D^\omega_a\phi \,D^\omega_b \phi  +U\phi^2) \,\zeta\,\omega_g\,. \]

When evaluated along this same section $s$ and after choosing appropriately the normalisation of $\zeta$, the quadratic form $\hat{\mathcal{E}}$ differs from $\mathcal{E}$ by the presence of a divergence term in the integrand, therefore leading to the same expression for the operator $-\Delta_{\mathcal E}$.

Although this provides a geometric framework in which the decoupling of the volume form $\mathfrak{m}$ from the metric arises naturally from the fact that the operator acts on sections of $\mathscr{G}[1]$, as opposed to (smooth) functions $f: M \to \R$, it introduces a gauge freedom that one might not consider as inherent to the original problem; consequently, we will also consider the problem without embedding it into a gauge theory. The two interpretations lead, of course, to different treatments of the inverse problem and distinct results.

The main question we wish to explore in this paper is what information about the metric $g$ and the measure $\mathfrak m$ (parametrised by $V$) can be extracted from the DN map, or more specifically, its full symbol when viewed as a pseudo-differential operator. 

The notion of DN map clearly depends on the type of object on which $-\Delta_{\mathcal{E}}$ acts. First of all, if $u$ is taken as a scalar function on $M$ as suggested by Eq.~\eqref{eq:quadratic}, then, based on the geometric data $(g,\mathfrak{m})$ of the problem, we might define the DN map to be:
\begin{equation}\label{eq:DNmap1} \Lambda^0 u_0 = \iota^*((\textrm{d}u)^\sharp \lrcorner \, \mathfrak{m}), \qquad u_0 \in C^\infty_0(\partial M),\end{equation}
where $\iota^*: \partial M \hookrightarrow \Mbar$ is the canonical inclusion map, $\star$ denotes the Hodge dual and $u$ solves the boundary value problem:
\begin{equation}\label{eq:DirichletPbDirect}\begin{cases} (-\Delta_g + g(dV,\cdot))u =0 & \textrm{on $M$}\,, \\
u = u_0 & \textrm{on $\partial M$}\,. \end{cases}\end{equation}

On the other hand, in the gauge-theoretic interpretation in which the operator is thought to act on sections $\phi\in\Gamma\big(\mathscr{G}[1]\big)$, it is natural to substitute the gauge covariant derivative $D^\omega \phi$ for $\textrm{d}u$ in the above formula.
The natural DN map will then be the operator:  \[ \Lambda^1: \Gamma\big (\iota^*\mathscr{G}[1]\big)\longrightarrow \Gamma \big(\iota^*\mathscr{G}[1]\otimes \Lambda^{n-1}\partial M\big),\]
defined by:
\begin{equation}
    \label{eq:DNMapGaugeNatural}
    \Lambda^1 \phi_0 = \iota^*\left((D^\omega\phi)^\sharp\lrcorner \,\omega_g\right)=\iota^*(\star D^\omega \phi)\,,
\end{equation}
and $\phi$ solves:
\begin{equation}\label{eq:DirichletGauge}\begin{cases} (-g^{ab}D_a^\omega D_b^\omega + U)\phi =0 & \textrm{on $M$}, \\
\phi = \phi_0 & \textrm{on $\partial M$}. \end{cases}\end{equation}
This geometric formulation defines the DN map as an object which can be tested against sections of $\iota^*\mathscr{G}[-1]$.
The advantage of this is that, exploiting the section $\zeta$, $\Lambda^1$ can be identified with an operator $\Lambda^1_{\zeta}$ given by
\begin{equation}\label{eq:DNmap2} \Lambda^1_{\zeta} \phi_0=\iota^*(D^\omega \phi)^\sharp \lrcorner \, \zeta\omega_g, \qquad \phi_0 \in \iota^*\Gamma\big(\mathscr{G}[1]\big),\end{equation}
and, expressed along an appropriate section $s$ (which leads to a distinguished section $e_s$ of $\mathscr{G}[1]$), we have: 
\[\zeta\,\omega_g=\mathfrak{m}\, e_s^{-2}.\]

Our approach to the inverse problems stemming from the above definitions is inspired by the classical computation due to Lee and Uhlmann~\cite{Lee_Uhlmann_1989} of the symbol of the DN-map for the Laplace-Beltrami operator $-\Delta_g$. This computation will be performed in boundary normal coordinates $x=(r,y^\mu)$ of the metric $g$, for which we recall the metric is of the form:
\[ g= \textrm{d}r^2 +g_{\alpha\beta}\textrm{d}y^\alpha\textrm{d}y^\beta\,. \]
We point out that these coordinates are naturally defined on a collar neighbourhood $\mathscr{C}= \psi([0,a)\times \partial M)$ of the boundary
 ($\psi$ is an embedding of $[0,a)\times \partial M$ into $\Mbar$ that restricts to the natural identification $\{0\} \times \partial M \rightarrow \partial M$).
 
 Throughout, Greek indices $\alpha, \beta, \dots \in \{1,\dots, n-1\}$ will be used to indicate components in the tangential directions to the boundary. When necessary, Latin indices $i,j,k..$ will range through $\{0,\dots,n-1\}$ with the understanding that $x^0=r, x^\alpha=y^\alpha.$

Following~\cite{Lee_Uhlmann_1989, Treves_1980}, we will also write:
\[ D_{x^j}=D_j=-i\partial_{x_j}, \qquad D_y=(D_{y^1}, \dots, D_{y^n})\,. \]

With these conventions and when adapted to our situation, the factorisation will take either of the following forms:
\begin{itemize}
\item If we embed the problem into a gauge theory:
\[ -g^{ij}(\partial_i - A_i)(\partial_j - A_j) +U \sim  (D_r +i A_r + iE(x)-iB(x,D_{y}))(D_r+iA_r+iB(x,D_{y}))\,,\]
where: \[ E= -\frac{1}{2}\,\partial_r \ln \delta\,.\]
This factorisation should be understood to be carried out \emph{after} writing the operator $\mathcal{S}$ along a section; $B(x,D_{y})$ will depend on this choice.
 \item Otherwise:
 \[ -\Delta_g+g(\textrm{d}V,\textrm{d})\sim (D_r+i \tilde{E} - iC(x,D_{y}))(D_r+iC(x,D_{y})), \]
 with: 
 \[ \tilde{E}=-\frac{1}{2}\partial_r\ln \delta + \partial_r V. \]
 \end{itemize}
In accordance with the point of view developed in~\cite[Chapter III]{Treves_1980}, it should be understood that $B(x,D_{y})$ and $C(x,D_{y})$ are families of pseudo-differential operators on $\partial M$, each of order at most $1$, that depend smoothly on the parameter $r$. This is consistent with the fact, pointed out previously, that we work in a collar neighbourhood $\mathscr{C}\simeq [0,a)\times \partial M$ of the boundary. 
Finally, in both cases the symbol $\sim$ is used to indicate equality up to addition of smoothing operators $\mathcal{E}'(\partial M) \rightarrow \mathcal{C}^\infty(\partial M)$ depending smoothly on $r$.

It is interesting to observe the structural differences in the dependence of $V$ in each case. These have an effect on how the Lee-Uhlmann procedure can be used to extract information about $V$ and $g$ and are at the heart of the differences we observe in our results.

The first problem we consider is the uniqueness of the couple $(g,V)$. In the computation of the symbol of either of the DN maps that we consider, the coupling between $g$ and $V$  turns out to be intricate. One of the motivations for introducing a gauge invariance was in fact to try to decouple this by a suitable choice of gauge, however, we were only able to extract the following:

\begin{prop}\label{prop:general_problem}
The symbol of the Dirichlet-to-Neumann map $\Lambda^1$ given by~\eqref{eq:DNMapGaugeNatural}-- or in fact, $\Lambda^1_\zeta$, see Eq. \eqref{eq:DNmap2}-- determines the boundary values of:
\begin{itemize}
    \item The metric components $g_{\alpha\beta}$ and their first radial derivatives $\partial_r g_{\alpha \beta}$,
    \item The weight $V$, up to an additive constant.
\end{itemize}  
\end{prop}

Proposition~\ref{prop:general_problem} does not hold for the DN map $\Lambda^0$ in Eq.~\eqref{eq:DNmap1}, essentially because the weight $V$ is present in $\tilde{E}$: this prevents us from even determining $\partial_r g_{\alpha\beta}$.
We will discuss this point further in Section~\ref{3.2}.

One can also question the uniqueness of one element of the pair $(g,V)$, assuming the other known; it is for these problems that the method is the most fruitful.

First, without much surprise, if the weight $V$ is known then, in either of our cases, the Lee-Uhlmann method (only with very minor computational differences) yields uniqueness of the metric $g$:
\begin{theo}\label{thm1}
If the weight $V$ is known, then one can determine the radial Taylor coefficients of the metric $g$ in boundary normal coordinates from either the symbol of the Dirichlet-to-Neumann map $\Lambda^1$ given by Eq. \eqref{eq:DNMapGaugeNatural} or the $\Lambda^0$ given by Eq.~\eqref{eq:DNmap1}.
\end{theo}
\noindent The proof of this theorem will therefore be omitted.

The more interesting problem is when $g$ is assumed known and the goal is to prove uniqueness of the weight. Here, there is a surprisingly stark difference between the two proposed interpretations of the operator $-\Delta_\mathcal{E}$. This is the content of:
\begin{theo}\label{thm2}
Assume the metric $g$ known, then the symbol of the Dirichlet-to-Neumann map $\Lambda^0$ defined in Eq.~\eqref{eq:DNmap1} determines uniquely the radial Taylor coefficients of the weight $V$.
\end{theo}
\noindent and:
\begin{theo}\label{thm3}
If the metric $g$ is known, then the symbol of the Dirichlet-to-Neumann map $\Lambda^1$, Eq. \eqref{eq:DNMapGaugeNatural}, or equivalently $\Lambda^1_\zeta$, Eq. \eqref{eq:DNmap2}, determines the boundary values of either:
\begin{itemize}
\item the radial Taylor coefficients of $V$, given the value of $\partial_r V$,
\item $\partial_r V$ and the radial Taylor coefficients of $V$ starting at order $3$, if instead the value of $\partial_r^2V$ is prescribed. In this case, $\partial_r V$ is determined by solving a second-order polynomial equation (which may have $2$ solutions), and the higher-order derivatives are completely determined by the choice of one of these roots.
\end{itemize}
\end{theo}
If we do not consider that the problem is embedded in a gauge theory then, in the real-analytic category, Theorem~\ref{thm2} shows that $V$ is uniquely determined by the DN map. This fails in the gauge-theoretic setting, where it appears that two weights might exist that will be very different. This is particularly surprising compared to, for instance, the results of Valero~\cite{Valero_2024} for spinors, where the more sophisticated algebraic structure of spinor fields provides an additional rigidity that is absent in our case.

It is interesting to observe that Theorems~\ref{thm2} and~\ref{thm3} can be slightly improved, namely the conclusions will be the same if we instead assume the knowledge of the Riemannian volume form $\omega_g$ in a neighbourhood of the boundary; this will be the substance of Theorems~\ref{thm4} and~\ref{thm5} in Sections~\ref{Section2.3} and \ref{3.2}. One might remark, however, that this assumption is slightly unnatural from the perspective of inverse problems.

We conclude our introduction by mentioning the important results of Kurylev, Oksanen and Paternain~\cite{Kurylev_Oksanen_Paternain_2018} on the reconstruction of a Riemannian manifold and a Hermitian vector bundle endowed with a compatible connection from the knowledge of the hyperbolic DN map associated with the wave equation arising from the connection Laplacian. The setting of our paper is at the same time different and simpler since we consider elliptic DN maps in the context of a rank $1$ real principal bundle. Furthermore, our results are limited to the recovery of the Taylor coefficients of the metric $g$ and weight $V$ along the boundary.

\section{Expressing $-\Delta_{\mathcal E}$ as a gauge-invariant Schr\"odinger operator}\label{Section2}
Our goal in this section is twofold. First, we explain in detail in Section~\ref{2.1} how the operator $-\Delta_{\mathcal E}$ can be interpreted in a gauge-theoretic context. Second, we obtain in Section~\ref{sec:gauge_computations} a factorisation of this operator in an arbitrary gauge, up to families of smoothing operators, enabling us to compute the full symbol of the DN map. This will in turn lead us, in Section~\ref{Section2.3}, to the proof of Proposition~\ref{prop:general_problem} and Theorem~\ref{thm3}.
\subsection{General considerations}\label{2.1}
Recall from \cite{Gonzalez-Lopez_Kamran_Olver_1994,Cotton_1900} that any second order differential operator $T$ given in local coordinates by:
\begin{equation} \label{eq:general_form}T= g^{ij}\frac{\partial^2}{\partial x^i \partial x^j} + b^i\frac{\partial}{\partial x^i}+c\,, \end{equation}
where $g^{ij},b^i,c$ are smooth real-valued functions on an open subset $\mathcal{O}\subset \Mbar$ and $(g^{ij})$ is assumed to define an everywhere  positive definite section of the bundle $S^2(\Mbar)$ of $(0,2)$ tensors, can be rewritten in the form
\begin{equation}\label{eq:general_forme2} T= g^{ij}(\nabla_i - A_i)(\nabla_i -A_j) + W\,, \end{equation}
where
\[ \begin{cases} A^i=g^{ij}A_j=-\frac{1}{2}b^i +\frac{1}{2}\delta^{-\frac{1}{2}}\partial_j(\delta^{\frac{1}{2}}g^{ij})\,, \\ W= c-A_iA^i +\delta^{-\frac{1}{2}}\partial_i(\delta^\frac{1}{2}A^i)\,, \\ \delta=\det(g_{ij}\,,) \end{cases}\]
and $\nabla$ is the Levi-Civita connection of $g$.
Thus, any operator of this type gives rise geometrically to a $(\R,+)$ gauge theory on the Riemannian manifold $(\Mbar,g)$, written in a gauge for which the local connection form is given by $-A_i \textrm{d}x^i.$ 

Indeed, let $\pi: P\rightarrow \Mbar$ be a principal $(\R,+)$-bundle over a Riemannian manifold (with boundary) $(\Mbar,g)$ and assume that $P$ is equipped with a principal connection $\omega$. Let $\mathscr{G}[1]$ be defined by equation~\eqref{def:G1}.

Any (local) section $s$ of $P$ determines a non-vanishing (positive\footnote{The bundle $\mathscr{G}[1]$ is canonically oriented.}) section $\tau_s$ of $\mathscr{G}[1]$, the induced linear connection $D^\omega$ on $\mathscr{G}$, this is locally determined by:
\[ D_b^\omega \tau_s = -A_b \tau_s,\]
where $A=-s^*\omega$.

Coupling $D^\omega$ with the Levi-Civita connection $\nabla$ on tensor products by the Leibniz rule, we can form a Schrödinger type operator:
\begin{equation} \label{eq:SchrodingerGauge}\mathcal{S}=-g^{ab}\,D^\omega_aD_b^\omega + U, \end{equation}
which locally acts on a section $\phi= u \tau_s$ by:
\[(-g^{ab}\,\nabla^\omega_a\,\nabla_b^\omega +U)\phi = (-g^{ab}(\nabla_a - A_a)(\nabla_b - A_b)u+Uu)\tau_s. \]

Observing that all $(\R,+)$-principal bundles are trivial and have a distinguished vertical vector field $n^a$ that generates the group action, we have the following characterisation of gauge theories that include the operator~\eqref{eq:operator_ambrosio}.
\begin{lemm}\label{lemm:characterisation}
Let $\pi: P \rightarrow \Mbar$ be a principal $\R$-bundle over a Riemannian manifold $(\Mbar,g)$ equipped with a principal connection $\omega$ and denote by $n^a$ the canonical vector field that generates the group action. Then, there is a section $s: \Mbar \rightarrow P$ along which the Schrödinger operator~\eqref{eq:SchrodingerGauge} takes the form given in Eq.~\eqref{eq:operator_ambrosio} if and only if there are smooth functions $V\in C^{\infty}(M), \lambda \in C^{\infty}(P)$ such that:
\[ \begin{cases} n^a\nabla_a \lambda=1,\\  \omega = -\frac{1}{2}\textnormal{d}(V\circ\pi) + \textnormal{d} \lambda,\\ U=-\frac{1}{2}\Delta_g V +\frac{1}{4}g(\textnormal{d}V,\textnormal{d}V).  \end{cases}\]
\end{lemm}

\begin{proof}
The conditions are sufficient as if $T=\Delta_{\mathcal{E}}$ then $T$ is of the form~\eqref{eq:general_form} with:
\[ c=0, \quad b^i= \delta^{-\frac{1}{2}}\partial_j(g^{ij}\delta^{\frac{1}{2}})-g^{ij}\nabla_j V, \]
hence can be written in the form~\eqref{eq:general_forme2} with: \[ A_i = \frac{1}{2}\nabla_i V, \quad W=-U=-\frac{1}{4}g(dV,dV)+\frac{1}{2}\Delta_g V.  \]
Observe now that $\{\lambda=0\}$ determines a distinguished global section $s$ of $P$, and the pullback of $\omega$ along this section is precisely $-\frac{1}{2}\textrm{d}V$.
Conversely, if such a section $s$ exists then if $\phi= v\tau_s$ and $A=s^*\omega$ we have:
\[ (-g^{ab}\,D_a^\omega\,D_b^\omega +U)\phi=(-\Delta_g v +2g^{ab}A_b\nabla_av+  (\nabla_a A^a-A_aA^a+ U)v)\tau_s \] 
Equating with Eq.~\eqref{eq:operator_ambrosio} we conclude that:
\[2A_b=\nabla_b V, \quad U=-\nabla_a A^a+A_aA^a.\]
Finally, $s$ determines a global trivialisation $P\simeq \Mbar\times \R$ and projection onto the second factor then yields a function $\lambda$ satisfying the required conditions.
\end{proof}
\begin{rema*} 
Observe that neither $V$ or $\lambda$ are uniquely determined. In particular, one may add an arbitrary constant to either function without affecting the conclusions of the lemma.
\end{rema*}
We also observe that by naturality of the exterior derivative $d$, these conditions behave well under the action of principal bundle automorphisms, that is we have: 
\begin{lemm}
Let $\Phi: P \rightarrow P$ be a principal bundle automorphism covering a diffeomorphism $\phi: \Mbar \rightarrow \Mbar$, then if $(g,\omega)$ satisfy the hypotheses of Lemma~\ref{lemm:characterisation} then so does $(\phi^*g, \Phi^*\omega)$.
\end{lemm}

The principal connection $\omega$ is flat, and therefore there is a section $\sigma: U \subset \Mbar \rightarrow P$ along which the operator can be expressed as the Schrödinger operator:
\[ -\Delta_g + U.\]
Indeed, let $s, V$ be fixed, then: $\sigma= s +\frac{1}{2}V$, will satisfy:
\[ \sigma^*\omega= 0,\]
the following Lemma is also an immediate consequence:
\begin{lemm}\label{lemm:ParallelGaugeMetric}
Let $\pi : P\rightarrow \Mbar$  and $\omega$ satisfy either of the equivalent conditions of Lemma~\ref{lemm:characterisation}, then 
up to multiplication by a positive constant, there exists a unique positive parallel section $\zeta$ of $\mathscr{G}[-2]$. More precisely, if $s: \Mbar\rightarrow P$ is a global section given by Lemma~\ref{lemm:characterisation} then, there is a constant $c\in \R$ such that:
\[ \zeta = e^{c-V}\tau_s^{-2}.\]
\end{lemm}

Observe that \[\tau_\sigma \equiv \zeta^{\frac{1}{2}}\]
 is a parallel nowhere vanishing section of $\mathscr{G}[1]$.
Since $V$ itself is at most determined up to addition of an arbitrary constant, there is no loss in generality in assuming that \[\zeta=e^{-V}\tau_s^{-2}.\]

\subsection{Factorisation in an arbitrary gauge} \label{sec:gauge_computations}
Now let $s$ be an \emph{arbitrary} section of $\pi: P\rightarrow \Mbar$ and $A_b$ the (local) connection form. Then evaluated along the section, $\mathcal{S}$ can be written:
\[ \begin{aligned} -\Delta^s_{g,U}:=-g^{ab}(\nabla_a - A_a)(\nabla_b - A_b)+U. \end{aligned}\]
Following~\cite{Lee_Uhlmann_1989}, we shall seek a factorisation of $\Delta^s _{g,U}$ , up to families of smoothing operators: $\mathcal{E}'(\partial M)\rightarrow C^\infty(\partial M)$ in the following form:
\begin{equation}\label{eq:DesiredFactorisation} -\Delta^s_{g,U} \sim (D_r +i A_r + iE(x)-iB(x,D_{y}))(D_r+iA_r+iB(x,D_{y}))\end{equation}
where $\sim$ means equality up to smoothing operators and: \[D_r=-i\,\partial_r, \qquad E= -\frac{1}{2}\partial_r \ln \delta.\] We emphasise that for each $r$, $B(x,D_{y})$ is a pseudo-differential operator of order at most $1$ on $\partial M$. Introducing:
\[Q=Q_2+Q_1+q_0,\] with:
\[\begin{cases}Q_2=g^{\alpha\beta}D_\alpha D_\beta,& %
\\ Q_1= -i\delta^{-\frac{1}{2}}\partial_\alpha(g^{\alpha\beta}\delta^{\frac{1}{2}})D_\beta +2iA^\beta D_\beta, & %
\\
q_0= U +\delta^{-\frac{1}{2}}\partial_\alpha(\delta^{\frac{1}{2}}A^\alpha)-A_\alpha A^\alpha, &  \end{cases}\]
it follows from Eq.~\eqref{eq:DesiredFactorisation} that $B$ must satisfy:
\begin{equation}\label{eq:DefiningB} i[D_r+i A_r,B]-Q+B^2-EB\sim 0. \end{equation}
Define $q_2,q_1,b$ to be the full symbols of $Q_2,Q_1,B$ respectively and let $\xi'=(\xi_\alpha)$ be the fibre coordinates on the cotangent bundle of $\partial M$, then using the symbol calculus, and in particular the fact that the symbol of the composite of two pseudo-differential operators $H,G$ (see~\cite[Theorem 4.3]{Treves_1980}) with symbols $h,g$ is given, up to smoothing operators, by:
\[ \sum_{K}\frac{1}{K!}\partial^K_\xi h D^K_xg,\]
we arrive at:
\begin{equation}\label{eq:DefbSymbol}-q_0 - q_1 -q_2-Eb +\partial_r b + \sum_{|K|\geq 1}\frac{1}{K!}\partial^K_{\xi'} bD^K_{y}A_r + \sum_{K}\frac{1}{K!}\partial^K_{\xi'} bD^K_{y} b \sim 0.\end{equation}
In particular, writing: \[ b(x,\xi') \sim \sum_{j\leq 1} b_j(x,\xi'),\]
and solving Eq.~\eqref{eq:DefbSymbol} by order of homogeneity in $\xi'$, we find that the sequence $(b_j)$ satisfies the following relations:
\begin{gather}\label{eq:Order1} ||\xi'||\equiv \sqrt{g^{\alpha\beta}\xi_\alpha\xi_\beta}=\sqrt{q_2}=- b_1,\\ \label{eq:Order0} 2b_1b_0= -\partial_r b_1 +q_1+Eb_1-\sum_{\alpha}\partial_{\xi_\alpha}b_1D_\alpha b_1, \\ \label{eq:Order-1}\begin{split}2b_{-1}b_1 =q_0 +Eb_0-\partial_rb_0-\sum_{\alpha} \partial_{\xi_\alpha}b_1D_\alpha A_r-b_0^2\\-\sum_\alpha\left(\partial_{\xi_\alpha} b_1D_\alpha b_0 + \partial_{\xi_\alpha}b_0D_\alpha b_1 \right),\end{split}\\\label{eq:recursivegeneric} \begin{split}2b_1b_j =-\partial_r b_{j+1} +Eb_{j+1}-\sum_{1\leq |K| \leq -j} \frac{1}{K!} \partial^K_{\xi'} b_{j+|K|+1}D^k_yA_r \\-\hspace{-2.5em} \sum_{\substack{j\leq m \leq 1\\0\leq |K| \leq m-j\\ (m,|K|)\notin\{ (j,0), (1,0)\}}}\hspace{-2.5em}\frac{1}{K!}\partial^K_{\xi'} b_m D^K_yb_{j+1+|K|-m}\end{split}  \quad (j\leq -2).\end{gather}
Each of the $b_j(x,\xi')$ is an expression that is positive homogeneous of degree $j$ in $\xi'$ (as can be seen by an immediate induction argument), hence $ \sum_{j\leq 1} b_j$ defines a classical (formal) symbol that determines a unique pseudo-differential operator up to smoothing operators, justifying the existence of the factorisation in the sense of pseudo-differential operators. 

Following~\cite{Lee_Uhlmann_1989,Valero_2024}, we point out that the principal symbol $c_1$ is chosen to be negative. This is to ensure the theory developed in~\cite[Chapter III]{Treves_1980}, and in particular Theorem 1.1 of that section, can be applied directly to the operator $B$ to prove Lemma~\ref{lemm:BoundaryOperator}, below. This lemma, which shows how to relate the symbol of the  DN map $\Lambda^1$ to that of  $B$ along the boundary $\partial M$, is the main motivation for studying the factorisation:
\begin{lemm}\label{lemm:BoundaryOperator}
Let $\phi\in\Gamma\big(\mathscr{G}[1]\big)$ be a solution of ~\eqref{eq:DirichletGauge} then if $\phi=u\tau$, $\phi_0=\phi\lvert_{\partial M}=u_0\tau$, $\tau\in\Gamma(\mathscr{G}[1])$, along an arbitrary section, then: \[(D_r+iA_r)u|_{\partial M} = -iB(x,D_{y})u|_{\partial M} + R\,u_0\,,\] for a smoothing operator $R$ on $\partial M$.
\end{lemm}
We refer to \cite{Lee_Uhlmann_1989} for the a proof in the case of Laplace-Beltrami operators and \cite{Valero_2024} for gauge Laplacians; our case is very similar, and the proof is omitted. 

\subsection{Discussion and proof of Proposition~\ref{prop:general_problem} and Theorem~\ref{thm3}}\label{Section2.3}

We are now ready to proceed with the proof of Proposition~\ref{prop:general_problem} and Theorems~\ref{thm1} and \ref{thm3}. Before doing so, we observe that whilst we have inserted our original problem and operator into a gauge theory, there is a sense in which the gauge symmetry is not completely arbitrary; indeed, there are at least two distinguished choices of gauge: 
\begin{itemize} \item  gauges in which the operator takes the desired form~\eqref{eq:operator_ambrosio} for some $V$,
\item  gauges in which the (local) connection form vanishes.
\end{itemize}
In these two cases, we have the following expressions:
\renewcommand{\arraystretch}{1.5}
\begin{table}[h!]
\centering
\begin{tabular}{c|c|c}
&$A_b = \frac{1}{2}\nabla_b V$ & $A_b=0$ \\\hline
 $q_1$ & $(-i\delta^{-\frac{1}{2}}\partial_\alpha(g^{\alpha\beta}\delta^{\frac{1}{2}})+ ig^{\alpha\beta}\partial_\alpha V)\xi_\beta$&$-i\delta^{-\frac{1}{2}}\partial_\alpha(g^{\alpha\beta}\delta^{\frac{1}{2}})\xi_\beta$\\\hline$q_0$&$-\frac{1}{2}\delta^{-\frac{1}{2}}\partial_r(\delta^{\frac{1}{2}}\partial_rV)+\frac{1}{4}(\partial_rV)^2$ &$-\frac{1}{2}\Delta_g V +\frac{1}{4}g(dV,dV)=U$ 
\end{tabular}
\caption{\label{t:symbolQ} Symbol of $Q$ according to gauge}
\end{table}

One might hope that having two distinguished choices of gauge will enable us to extract more information than in the general case where the potential $U$ and the connection are completely decoupled. However, this expectation is not met (observe in particular that the difference between the expressions of $q_1$ and $q_0$ only involves transverse derivatives of $V$). Instead, the couple $(g,V)$ is uniquely determined only up to first order as stated in Proposition~\ref{prop:general_problem}, which we now prove:

\begin{proof}[Proof of Proposition~\ref{prop:general_problem}]
First, let us make an arbitrary choice of parallel section $\zeta$ to transform the DN map into the form~\eqref{eq:DNmap2}. Let us assume that there is a gauge in which the operator takes the desired form~\eqref{eq:operator_ambrosio} for some fixed $V$; the connection form is then $\omega=-\frac{1}{2}\textrm{d}V$. Observe that this gauge is fixed up to constant global gauge transformations: $s\mapsto s + \alpha$. Fixing once and for all a choice, we obtain a non-vanishing section $\tau_s$ of $\mathscr{G}[1]$, we will then have: $\zeta \tau^2_s=e^{-V+c}$, for some arbitrary constant $c$. Since $V$ itself can only be determined up to arbitrary constants, it is no loss of generality in setting $c=0$.

Working in the chosen gauge $s$, projected out on boundary normal coordinates, Eq. \eqref{eq:DNmap2} then becomes:
\[ \Lambda^1_\zeta \phi_0=(\partial_ru - \frac{1}{2}\partial_rVu)|_{\partial M} e^{-V}\sqrt{\delta}\tau_s^{-1} \,dy^1\wedge \dots \wedge dy^{n-1}\,,\]
where we have set $\phi=u\tau_s$ and $\phi$ solves Eq.~\eqref{eq:DirichletGauge}.

Applying Lemma~\ref{lemm:BoundaryOperator} in our gauge, we have the following relationship between the DN map and the operator $B$ constructed for the factorisation:
\[\Lambda^1_\zeta\phi_0\sim B_1(x,D_y)\phi|_{\partial M}e^{-V}\sqrt{\delta}\tau_s^{-1}\, dy^1\wedge \dots \wedge dy^{n-1}, \]

On the other hand, one might choose to work in the gauge $\sigma=\zeta^{\frac{1}{2}}$ where the local connection form is vanishing, and in this case the factorisation will lead to a relation:
\[ \Lambda^1_\zeta \phi_0 =(\partial_r v)|_{\partial M}\sqrt{\delta}\tau_\sigma^{-1}\,dy^1\wedge \dots \wedge dy^{n-1}\sim B_2(x,D_y)v|_{\partial M}\sqrt{\delta}\tau_\sigma^{-1}dy^1\wedge \dots \wedge dy^{n-1}, \]
where $\phi = v\tau_\sigma$ is a solution of Eq.~\eqref{eq:DirichletGauge}.

Now, the principal symbol of the operators $B_1$ and $B_2$ are identical and given by $b_1=-\lVert\xi'\rVert$. Hence the principal symbol of the DN map, expressed in the gauge $\sigma$, will enable us to determine the bilinear form $g^{\alpha\beta}\delta$. Taking the determinant, we find:
$\delta^{n-2}$ and hence obtain the boundary value of the inverse metric $g^{\alpha\beta}$.
Returning then to the expression in the gauge $s$, we can determine the boundary value of the weight $V$, (with the constant fixed arbitrarily by our previous choices.)
In both cases, it then follows that we can also determine the boundary values of all the other coefficients $(b_j)_{j\leq 1} $ in the expansions of either of $B_1$ and $B_2$.

With this data, we now look to~\eqref{eq:Order1}-\eqref{eq:recursivegeneric} to find expressions for the radial derivatives of $V$ and $g$.  In the gauge $\sigma$, the weight $V$ does not intervene at all in Eq.~\eqref{eq:Order0} and the terms can be rearranged to obtain the relationship given in~\cite[Eq. (1.10)]{Lee_Uhlmann_1989} by setting:
\[h^{\alpha\beta}=\partial_r g^{\alpha\beta}, \quad h=h^{\alpha\beta}g_{\alpha\beta}=2E, \quad k^{\alpha\beta}=h^{\alpha\beta}-hg^{\alpha\beta},\]
that is:
\begin{equation}\label{eq:order0c} b_0=-\frac{1}{4\lVert\xi'\rVert^2} k^{\alpha\beta}\xi_\alpha\xi_\beta +T(g^{\alpha\beta})\,, \end{equation}
where $T$ depends on $\xi'$, $g^{\alpha\beta}$ and its \emph{tangential} derivatives.
Repeating the reasoning carried out in~\cite{Lee_Uhlmann_1989}, we can therefore determine $h^{\alpha\beta}=\partial_rg^{\alpha\beta}$. 
\end{proof}

Note that unfortunately and despite the promising start illustrated in Proposition~\ref{prop:general_problem}, the procedure stops as soon as the coupling between $g$ and $V$  arises in the equations governing the induction, preventing us from continuing to solve for their higher-order radial derivatives. Indeed, working in either gauge and moving to the next step in the induction, see Eq.~\eqref{eq:Order-1},  we obtain a single relationship involving $\partial_r V, \partial^2_rV$ and $\partial^2_rg^{\alpha\beta}$. For instance, using the section $\sigma$ in which the connection form vanishes, this relationship is:
\[2b_{-1}b_1=U+\frac{1}{4\lVert \xi'\rVert^2}\partial_rk^{\alpha\beta}\xi_\alpha\xi_\beta +R\,,\]
where the remainder $R$ only depends on $g^{\alpha\beta}$, its first radial derivative and tangential derivatives.
The expression~\eqref{eq:defU} of the potential $U$ explicitly involves first and second radial derivatives of $V$, as claimed.

One might hope to gain an additional equation by working simultaneously along the sections $s$ and $\sigma$. Nevertheless, up to tangential derivatives, the relationships are in fact identical. For instance, using the notations from Table~\ref{t:symbolQ}, we see that:
\[\begin{aligned} U&= -\frac{1}{2}\delta^{-\frac{1}{2}}\partial_i\left(\delta^{\frac{1}{2}}g^{ij}\partial_j\right)V +\frac{1}{4}g^{ij}\partial_iV\partial_jV,\\&= q_0-\frac{1}{2}\delta^{-\frac{1}{2}}\partial_\alpha\left(\delta^{\frac{1}{2}}g^{\alpha\beta}\partial_\beta\right)V+\frac{1}{4}g^{\alpha\beta}\partial_\alpha V \partial_\beta V \,.\end{aligned} \]

In view of this, we will therefore consider inverse problems for which one of either $V$ or $g$ is given; these cases are covered in Theorems~\ref{thm1} and~\ref{thm3}. As observed in the introduction, the proof of Theorem~\ref{thm1} is a minor modification of Lee and Uhlmann's original argument and is omitted. We now focus on the proof of Theorem~\ref{thm3}, which concerns the case in which the metric $g$ is given and $V$ is to be recovered and in which there is a surprising indetermination.

\begin{proof}[Proof of Theorem~\ref{thm3}]
Repeating the first steps of the proof of Proposition~\ref{prop:general_problem} we can recover the values of the weight $V$ along the boundary $\partial{M}$ up to an additive constant, provided we work in a gauge in which $A_b=\frac{1}{2}\nabla_b V$, which we assume to exist. We remain in such a gauge for the remainder of the computation. 

The second step $j=0$ (see ~\eqref{eq:Order0}) does not provide any new information: no new radial derivatives of $V$ arise since the gauge invariant principal symbol is not affected by the weight $V$. However, using the expression for $q_1$ in the first column of Table~\ref{t:symbolQ}, we may rewrite this equation:
\[b_0 = \frac{ig^{\alpha\beta}\partial_\alpha V\xi_\beta}{2b_1}+ R\,,\]
where the remainder term $R$ is completely independent of $V$.
In particular,
\[\partial_r b_0= \frac{ig^{\alpha\beta}\partial_\alpha A_r \xi_\beta}{b_1}+\tilde{R}\,,\]
where, again, the remainder term $\tilde{R}$ does not depend on $V$. Furthermore:
\[ \partial_{\xi_\alpha}b_1=-\frac{g^{\alpha\beta}\xi_\beta}{\lVert \xi'\rVert}=\frac{g^{\alpha \beta}\xi_\beta}{b_1}\,,\]
so that:
\begin{equation}\label{eq:simplification} \partial_rb_0= -\sum_{\alpha}\partial_{\xi_\alpha}b_1 D_\alpha A_r +\tilde{R}\,.  \end{equation}

The radial derivatives of $V$ surface at the step $j=-1$ in Eq.~\eqref{eq:Order-1}, that we reproduce below:
\[\begin{split}2b_{-1}b_1 =q_0 +Eb_0-\underline{\partial_rb_0-\sum_{\alpha} \partial_{\xi_\alpha}b_1D_\alpha A_r}-b_0^2-\sum_\alpha\left(\partial_{\xi_\alpha} b_1D_\alpha b_0 + \partial_{\xi_\alpha}b_0D_\alpha b_1 \right).\end{split}\]  In this expression, the coefficient $q_0$ and, a priori, the connection coefficient $A_r$ and $\partial_rb_0$ will all contain radial derivatives of $V$. However, by Eq.~\eqref{eq:simplification} the underlined term above does not in fact depend on $V$. Hence, only $q_0$ contributes radial derivatives of $V$. Using its expression and rearranging the terms we can write:
\[ b_{-1} = \frac{1}{4 b_1}(-\partial_r^2V+E\partial_rV +\frac{1}{2}(\partial_rV)^2)+R\,,\]
where the remainder term depends on the metric and the tangential derivatives $\partial_\alpha V$.
This equation is the basis for the dichotomy in the result, as two radial derivatives of $V$ arise simultaneously with only one free equation. For the higher-order derivatives on the other hand, Eq.~\eqref{eq:recursivegeneric} enables us to show by induction that for $j\leq -2$:
\[ b_{j}=\frac{1}{2}(-2b_1)^j\partial_r^{-j+1}V + R_j\, \]
where in the remainder $R_j$ the order of any radial derivatives of $V$ is at most $-j$, which shows that the rest of the radial derivatives of $V$ can be determined iteratively from the Dirichlet to Neumann map.
\end{proof}

As mentioned in the introduction, this result can be improved in the following manner:
\begin{theo}\label{thm4}
Suppose that the Riemannian volume form $\omega_g$ is known in a neighbourhood of the boundary. Then, along the boundary, the radial Taylor coefficients in boundary normal coordinates of the metric $g$ are uniquely determined by the Dirichlet-to-Neumann map $\Lambda^1$ and
those of $V$ are determined, provided either $\partial_r V$ or $\partial_r^2V$ is prescribed. In the latter case, $\partial_rV$ is calculated by solving a second order polynomial equation which may have two distinct roots.
\end{theo}

\begin{proof}
We have already established in Proposition~\ref{prop:general_problem} that without any additional hypotheses we can always determine the boundary value and the first radial derivative of $g^{\alpha\beta}$. We resume the proof at this stage with our new assumption; observing that in boundary normal coordinates the volume form is: \[\omega_g = \sqrt{\delta}\,dr\wedge dy^1\dots\wedge dy^{n-1},\]
we can conclude that the function $\delta$ is now known.
Additionally, since:
\[ \partial_r\ln \delta = g^{\alpha\beta}\partial_r g_{\alpha\beta}=-h^{\alpha\beta}g_{\alpha\beta}=-h\,,\]
the trace of any of the radial derivatives of $g$ is also given from the outset.

We recall that we can work on one of two distinguished gauges and the computations will lead to the same results. Here, we resume our computations in the gauge $\sigma$ introduced in the proof of Proposition~\ref{prop:general_problem}, for which the connection form vanishes.
Previously, the identity: \[2b_{-1}b_1=U+\frac{1}{4\lVert \xi'\rVert^2}\partial_rk^{\alpha\beta}\xi_\alpha\xi_\beta +R\,,\]
could not be solved to determine all of the three unknowns, i.e. the boundary values of $\partial_r k^{\alpha\beta}, \partial_rV$ and $\partial_rV^2$. However, here, let us rewrite this equation as follows.
\[\begin{aligned}2b_{-1}b_1&=-\frac{1}{2\sqrt{\delta}}\partial_r(\sqrt{\delta}\partial_r V)+ \frac{1}{4}(\partial_r V)^2+ \frac{1}{4\lVert \xi'\rVert^2}\partial_rk^{\alpha\beta}\xi_\alpha\xi_\beta +R\,,
\\&=  \frac{1}{4\lVert \xi'\rVert^2}\left(\partial_rk^{\alpha\beta} + ((\partial_r V)^2-2\partial^2_rV)g^{\alpha\beta}
\right)\xi_\alpha\xi_\beta + R\,.\end{aligned}\]
We point out at that between each equation the remainder term $R$ has changed: we have moved all the known terms into it, in particular, all terms involving only $\delta$ and its derivatives.

It follows that one can determine the quadratic form:
\[ \partial_rk^{\alpha\beta} + ((\partial_r V)^2-2\partial^2_rV)g^{\alpha\beta}\,.\]
Taking the trace and exploiting the fact that $g_{\alpha\beta}\partial_rk^{\alpha\beta}$ is now known, we obtain an equation with only two unknowns: $\partial_rV$ and $\partial^2_rV$.

In the problem of determining the Taylor coefficients of $V$, we recover the dichotomy of Theorem~\ref{thm3}. On the other hand, in that of determining those of $g$, we can now solve for $\partial_rk^{\alpha\beta}$.

Repeating this argument inductively, we can determine all of the radial Taylor coefficients of $g^{\alpha\beta}$ and given a solution to the coupled equation for the first and second derivative of $V$, we can obtain all its radial derivatives from order $3$ onwards. 

\end{proof}

\section{The uniqueness of $V$ when $-\Delta_{\mathcal E}$ acts on $C^{\infty}(\Mbar)$}
In the gauge-theoretic context considered in Section~\ref{Section2}, we were not able to answer positively the question of uniqueness of $V$ in the inverse problem for which the metric $g$ is given: the information contained in the first radial derivative of $V$ along the boundary $\partial M$ was missing in the symbol of the DN map defined in Eq.~\eqref{eq:DNMapGaugeNatural}. In the current section, we revisit the same inverse problem taking instead the point of view in which $-\Delta_{\mathcal E}$ acts on $C^{\infty}(\Mbar)$, as opposed to sections of the bundle $\mathscr{G}[1]$. In contrast, we shall see that the symbol of the corresponding DN map given by Eq.~\eqref{eq:DNmap1} determines uniquely all the radial Taylor coefficients of the weight $V$ along the boundary $\partial M$.

We perform in Section~\ref{Newfactorisation} the necessary adaptations to the factorisation of $-\Delta_{\mathcal E}$, taking into account the different definition of the DN map, see Eq.~\eqref{eq:DNmap1}. Then, in Section~\ref{3.2}, we give the proof of Theorem~\ref{thm2}.

\subsection{Factorisation}\label{Newfactorisation}
When the DN map is defined by Eq.~\eqref{eq:DNmap1}, the following factorisation is natural: 
\begin{equation}\label{Eq:22} 
-\Delta_g+g(\textrm{d}V,\textrm{d})\sim (D_r+i \tilde{E} - iC)(D_r+iC)=D_r^2 +i\tilde{E}D_r+i[D_r,C]-\tilde{E}C+C^2\,,
\end{equation}
There is a manifest structural difference with the previous case in that the dependence on $V$ is now contained in the coefficient $\tilde{E}$ which is given by:
\[\tilde{E}=-\frac{1}{2}\partial_r\ln \delta + \partial_r V.\]
To determine $C$, we use that:
\begin{equation}\label{relations}
\begin{aligned} -\Delta_g + g(\textrm{d}V ,d) &= \frac{1}{\sqrt{\delta}}\,D_i(\sqrt{\delta}\,g^{ij})D_j +ig^{ij}\partial_i V D_j\\&=
D_r^2+ (-\frac{1}{2}\partial_r\ln\delta +\partial_rV)iD_r +\frac{1}{\sqrt{\delta}}\,D_\alpha(\sqrt{\delta}\,g^{\alpha\beta}D_\beta) +ig^{\alpha\beta}\partial_\alpha V D_\beta\\&= D_r^2+ i\tilde{E}D_r +Q\,, \end{aligned}\end{equation}
with: \[ Q=\frac{1}{\sqrt{\delta}}\,D_\alpha(\sqrt{\delta}\,g^{\alpha\beta}D_\beta) +ig^{\alpha\beta}\partial_\alpha V D_\beta,\]
which leads to the following:
\[i[D_r,C]+C^2-\tilde{E}C-Q\sim 0.\]

Assuming the symbol $c$ of $C$ satisfies $\displaystyle c\sim \sum_{j\leq 1} c_j(x, \xi')$, where $\xi'=(\xi_\alpha)$ are coordinates in the fibre directions on  $T\partial M$, we deduce the relations:
\begin{equation}\label{recurrence_c} \begin{cases} c_1^2 = g^{\alpha\beta}\xi_\alpha\xi_\beta\,, \\ \partial_r c_1 +2c_0c_1 +\displaystyle \sum_{\alpha} \partial_{\xi_\alpha}c_1 D_\alpha c_1 -\tilde{E}c_1+(\frac{1}{\sqrt{\delta}}\,\partial_\alpha(\sqrt{\delta} \,g^{\alpha\beta})-g^{\alpha\beta}\partial_\alpha V)\xi_\beta=0\,,\\ 2c_{j-1}c_1 + \partial_r c_j -\tilde{E}c_j + \displaystyle \sum_{\substack{0\leq|K|\leq m-j+1\\ j-1\leq m \leq 1\\ (|K|,m)\notin\{ (0,j-1), (0,1)\}}} \frac{1}{K!}\partial^K_{\xi'}c_m D^K_{x'}c_{j+|K|-m}=0\,,&j\leq 0\,. \end{cases} \end{equation}
As before, each of the $c_j(x,\xi')$ is an expression that is positive homogeneous of degree $j$ in $\xi'$ (as can be seen by an immediate induction argument), hence $ \sum_{j\leq 1} c_j$ defines a classical (formal) symbol that determines a unique pseudo-differential operator up to smoothing operators, justifying the existence of the factorisation in the sense of pseudo-differential operators. 

As in Section~\ref{sec:gauge_computations}, the principal symbol is chosen to be the negative square root $c_1=-\sqrt{g^{\alpha\beta}\xi_\alpha\xi_\beta}$ and the factorisation is exploited to relate the symbol of the DN map $\Lambda^0$ to that of $C$ along the boundary $\partial M$ by: 
\begin{lemm}
Let $u\in H^1(M)$ be a solution of Eq.~\eqref{eq:DirichletPbDirect}, $u|_{\partial M}=u_0\in H^{1/2}(\partial M)$, then: \[D_r u|_{\partial M} = -iC(x,D_{y})u|_{\partial M} + R\,u_0\,,\] for a smoothing operator $R$ on $\partial M$.
\end{lemm}
Once more, the proof of this lemma is omitted as it follows very closely the arguments given in~ \cite{Lee_Uhlmann_1989} for the Laplace-Beltrami operators and in \cite{Valero_2024} for gauge Laplacians on spinors.

\subsection{Discussion and proof of Theorem~\ref{thm2}}\label{3.2}
Before discussing the proof of Theorem~\ref{thm2}, we shall briefly consider the more general inverse problem in which we wish to determine the couple $(g,V)$. It turns out that the situation is worse than in the gauge-theoretic setting: even their boundary values cannot be determined.

To see this, recall that the appropriate DN map $\Lambda^0$ is defined by Eq.~\eqref{eq:DNmap1} and, in boundary normal coordinates, can be expressed along the boundary as:
\[\Lambda^0 u_0 = \partial_r u|_{\partial M}e^{-V}\sqrt{\delta}dy^1\wedge\dots\wedge dy^{n-1}\sim C(x,D_y)u|_{\partial M}e^{-V}\sqrt{\delta}dy^1\wedge\dots\wedge dy^{n-1}, \]
where $u$ solves Eq.~\eqref{eq:DirichletPbDirect}.

Here, squaring the principal symbol of $\Lambda^0$ will enable us to determine $g^{\alpha\beta}\sqrt{\delta}e^{-V}$.
However, the determinant of this now involves both $\delta$ and $e^{-V}$ and cannot be used to determine either of them.

We now return to the main focus of this section: the inverse problem in which $g$ is known.

\begin{proof}[Proof of Theorem 2]
The factorisation~\eqref{Eq:22} shows that, up to smoothing operators we have:
\[\Lambda^0 \sim (C e^{-V}\delta^{\frac{1}{2}})|_{\partial M} dy^1\wedge \dots \wedge dy^n.\]

Since the principal symbol $c_1$ of $C$ is given by $-\sqrt{g^{\alpha\beta}\xi_\alpha\xi_\beta}$, hence known from the outset, we can immediately determine the boundary value of the weight $V$. It then follows that we can determine the boundary value of any of the coefficients $c_i(x,\xi')$, of the formal symbol.

For the next step, we begin by rewriting the second equation in Eq.~\eqref{recurrence_c}, using $\tilde{E}= -\frac{1}{2}\partial_r\ln\delta +\partial_rV$,  as follows:
\begin{equation}\label{eq:c0}\begin{aligned} 2c_0c_1 -\frac{1}{2||\xi'||}(h^{\alpha\beta}-(h+2\partial_r V)g^{\alpha\beta})\xi_\alpha\xi_\beta &+\displaystyle \sum_{\alpha} \partial_{\xi_\alpha}c_1 D_\alpha c_1 \\ &+(\frac{1}{\sqrt{\delta}}\,\partial_\alpha(\sqrt{\delta}\, g^{\alpha\beta})-g^{\alpha\beta}\partial_\alpha V)\xi_\beta=0.\end{aligned}\end{equation}
Once more, we we have adopted the notations introduced in~\cite{Lee_Uhlmann_1989}: 
 \[h^{\alpha\beta}=\partial_r g^{\alpha\beta}\,,\qquad h=h^{\alpha\beta}g_{\alpha\beta}\,,\qquad ||\xi'||=\sqrt{g^{\alpha\beta}\xi_\alpha\xi_\beta}=-c_1\,.\]
 
Rearranging the terms in Eq.~\eqref{eq:c0} to isolate $\partial_r V$, we arrive at:
\[c_0 = \frac{1}{2}\partial_rV + R(g, T_0(V))\,,\]
where $R$ only depends on $g$, the boundary value of $V$ and its tangential derivatives.

We claim that when $j\leq 0$:
\[ c_j =(-1)^j(2c_1)^{j+1}\partial_r^{-j+1}V + \tilde{R}_j(g,T_j(V))\,,\]
where, this time, the remainder depends on $g$ the transverse derivatives of $V$ up to order $-j$ and their tangential derivatives.
Indeed, if it holds up to some $c_j$, then using the recurrence relation~\eqref{recurrence_c}, we can rearrange the terms (isolating this time the normal derivative $\partial_r c_j$) so that:
\[ c_{j-1}=-\frac{1}{2c_1}\partial_r c_j+R(g,T_j(V))\,, \]
for some remainder $R$ that contains only transverse derivatives of $V$ up to order $-j+1$, the result then follows by derivation of $c_j$.
Hence all the radial derivatives of $V$ can be obtained iteratively from the boundary values of the formal symbol of $C$.
\end{proof}

As mentioned in the introduction, we can improve the above result in the following manner:
\begin{theo}\label{thm5}
Suppose that the Riemannian volume form $\omega_g$ is known in a neighbourhood of the boundary. Then, along the boundary, the radial Taylor coefficients in boundary normal coordinates of both the potential $V$ and the metric $g$ are uniquely determined by the Dirichlet-to-Neumann map $\Lambda^0$.
\end{theo}
\begin{proof}
In boundary normal coordinates the volume form is: \[\omega_g = \sqrt{\delta}\,dr\wedge dy^1\dots\wedge dy^{n-1}\,,\]
thus, the additional assumption translates to the knowledge of the function $\delta$. Moreover:
\[ \partial_r\ln \delta = g^{\alpha\beta}\partial_r g_{\alpha\beta}=-h^{\alpha\beta}g_{\alpha\beta}=-h\,,\]
where we have used the same notations as above. This shows that $h$ and all of its derivatives are in fact known from the outset.

Since $\delta$ is known, the obstruction to determining the boundary values of $g_{\alpha\beta}$ and $e^{-V}$ from the principal symbol of the DN map that we highlighted at the beginning of the section has been lifted: we can now determine the boundary value $e^{-V}$ from the determinant of the square of the principal symbol $q(\xi)=\delta \,e^{-2V} g^{\alpha\beta}\xi_\alpha\xi_\beta$, and then recover the boundary value of $g^{\alpha\beta}$. As before, this then can be used to determine the boundary values of all the coefficients $c_j(x,\xi')$ in the formal expansion of $C$.

Equation~\eqref{eq:c0} can now be exploited to determine $h^{\alpha\beta}$ and $\partial_r V$. Indeed, rearranging the terms we obtain:
\[ c_0= -\frac{1}{4\lVert\xi'\rVert^2}\tilde{k}^{\alpha\beta}\xi_{\alpha}\xi_{\beta} + R\,,\]
where the remainder only depends on the boundary values of $g$ and $V$ and/or their tangential derivatives and:
\[ \tilde{k}^{\alpha\beta}=h^{\alpha\beta}-(h+2\partial_r V)g^{\alpha\beta}.\]
Hence, $\tilde{k}^{\alpha\beta}$ can be determined and since $h$ is known, its trace $\tilde{k}=k^{\alpha\beta}g_{\alpha\beta}$ can be used to find: 
\[ \partial_r V = -\frac{\tilde{k}+(n-2)h}{2(n-1)}\,. \]
As before, the relationship between $\tilde{k}^{\alpha\beta}$ and $h^{\alpha\beta}$ can be reversed:
\[ h^{\alpha\beta} = \tilde{k}^{\alpha\beta} + \left(h -\frac{\tilde{k}+(n-2)h}{2(n-1)}\right)g^{\alpha\beta}, \]
and we can compute $h^{\alpha\beta}$.
The argument can be repeated inductively to obtain the boundary values of all of the radial derivatives.
\end{proof}

\subsection*{Acknowledgements}
Research supported by NSERC grant RGPIN 105490-2018. We thank Thierry Daud\'e for helpful comments and suggestions.
\nocite{*}
\printbibliography[title=Bibliography]

 \end{document}